\documentclass[a4paper,11pt]{amsart}
\usepackage{tikz,tikz-cd}
\usetikzlibrary{babel, calc, fadings, intersections}
\usepackage{mathpazo}
\usepackage{enumitem}
\usepackage[dvipsnames]{xcolor}
\usepackage{mathrsfs}  
\usepackage{bbm} 
\usepackage{fancyhdr}
\usepackage{amsmath, amsfonts, amssymb, amsthm, tensor, mathabx}
\usepackage[all]{xy}
\usetikzlibrary{decorations.pathreplacing,calligraphy}
\usepackage[T1]{fontenc}
\usepackage[utf8]{inputenc}
\usepackage{graphicx}
\usepackage[vmargin=3cm, hmargin=3cm]{geometry}
\usepackage[colorlinks = true,
            linkcolor = blue,
            urlcolor  = blue,
            citecolor = red,
            anchorcolor = blue,
pdftitle={},
pdfauthor={},
linktocpage=true
]{hyperref}
\usepackage{cleveref}
\usepackage[scr=rsfs]{mathalpha}

\usepackage[mathcal]{eucal}

\DeclareFontFamily{OT1}{pzc}{}
\DeclareFontShape{OT1}{pzc}{m}{it}{<-> s * [1.10] pzcmi7t}{}
\DeclareMathAlphabet{\mathpzc}{OT1}{pzc}{m}{it}

\usetikzlibrary{decorations.pathreplacing,calc}
\usetikzlibrary{arrows.meta,bending,quotes}
\usetikzlibrary{arrows}

\usepackage{mathtools}
\def\co{\colon\thinspace}

\mathchardef\mhyphen="2D

\numberwithin{equation}{section}

\newcommand{\Bb}{\mathcal{B}}
\newcommand{\C}{\mathbb{C}}
\newcommand{\R}{\mathbb{R}}
\newcommand{\N}{\mathbb{N}}
\newcommand{\Z}{\mathbb{Z}}
\newcommand{\Q}{\mathbb{Q}}
\renewcommand{\P}{\mathbb{P}}
\newcommand{\T}{\mathbb{T}}
\newcommand{\bend}{\mathpzc{B}}
\newcommand{\Vtrop}{\textup{trop}}
\newcommand{\Supp}{\textup{Supp}}
\newcommand{\Strop}{\mathpzc{Trop}}

\newcommand{\Span}{\operatorname{span}}
\newcommand{\im}{\operatorname{Im}}

\newcommand{\trop}{\operatorname{trop}}

\DeclareMathOperator{\Hom}{Hom}
\DeclareMathOperator{\Spec}{Spec}
\DeclareMathOperator{\an}{an}
\DeclareMathOperator{\ev}{\mathpzc{ev}}

\newtheorem{theorem}{Theorem}[subsection]  
\newtheorem{lemma}[theorem]{Lemma} 
\newtheorem{proposition}[theorem]{Proposition}
\newtheorem{corollary}[theorem]{Corollary}
\newtheorem{conjecture}[theorem]{Conjecture}
\newtheorem{question}[theorem]{Question}
\theoremstyle{remark} 
\newtheorem{definition}[theorem]{Definition}
\newtheorem{remark}[theorem]{Remark}
\newtheorem{Exa}[theorem]{Example}

\title[Three lectures on tropical algebra]{Three lectures on tropical algebra}
 
\author[J. Giansiracusa]{Jeffrey Giansiracusa}

\author[K. K\"{u}hn]{Kevin K\"{u}hn}

\author[S. Mereta]{Stefano Mereta}

\author[E. Vital]{Eduardo Vital}

\AtEndDocument{\bigskip{\footnotesize%
\textsc{Jeffrey Giansiracusa, Durham University, UK} \par\vspace{-7pt}  
  \addvspace{\medskipamount}
  \textit{E-mail:} \texttt{jeffrey.giansiracusa@durham.ac.uk} \par
  \addvspace{\medskipamount}
\textsc{Kevin K\"{u}hn, TU Berlin, Germany} \par\vspace{-7pt}  
  \addvspace{\medskipamount}
  \textit{E-mail:} \texttt{kuehn@math.tu-berlin.de} \par
  \addvspace{\medskipamount}
\textsc{Stefano Mereta, CUNEF Universidad, Madrid, Spain} \par\vspace{-7pt}  
  \addvspace{\medskipamount}
  \textit{E-mail:} \texttt{stefano.mereta@cunef.edu} \par
  \addvspace{\medskipamount}
\textsc{Eduardo Vital, Universit\"{a}t Bielefeld, Germany} \par\vspace{-7pt}
  \textit{E-mail:} \texttt{evital@math.uni-bielefeld.de}
}}

\begin{document}
\date{\today \\ 
K. K\"{u}hn received support
by Deutsche Forschungsgemeinschaft (DFG) through ''Symbolic Tools in Mathematics and their
Application'' (TRR $195$, project-ID \texttt{286237555}). S. Mereta was partially supported by the Wallenberg AI, Autonomous Systems and Software Program (WASP) funded by the Knut and Alice Wallenberg Foundation. E. Vital was funded by the Deutsche Forschungsgemeinschaft (DFG, German Research Foundation) – TRR $358$, Project-ID \texttt{491392403}.}

\begin{abstract}
This document is a slightly expanded version of a series of talks given by J. Giansiracusa at the workshop `Geometry over semirings' at Universitat Aut\`{o}noma de Barcelona in July 2025.  In the first lecture  we introduce tropical polynomials, ideals, congruences, and how the connection with tropical geometry is made via congruences of bend relations. Tropical geometry and matroid theory are telling us that we should focus attention on a narrow slice of the world of tropical algebra, and this leads to the theory of tropical ideals (as developed by Maclagan and Rinc\'{o}n) and an abundance of interesting open questions.  In the second lecture we examine the relationship between Berkovich analytification and tropicalization from the perspective of bend relations, giving a refinement of Payne's influential limit theorem.  In the third lecture we set aside geometry and focus on tropicalization via bend relations as a construction in commutative and non-commutative algebra.  Constructions such as symmetric algebras, exterior algebras, matrix algebras, and Clifford algebras can be tropicalized.  In the case of exterior algebras, the resulting tropical notion beautifully completes the picture of the Pl\"{u}cker embedding and gives a new perspective on the tropical Pl\"{u}cker relations. For matrix algebras and Clifford algebras, Morita theory becomes an interesting topic.
\end{abstract}

\maketitle

\section*{Introduction}
    \label{sec:introduction}
This document is a slightly expanded version of a series of talks given by J. Giansiracusa at the workshop `Geometry over semirings' at Universitat Aut\`{o}noma de Barcelona in July 2025.

\subsection*{First lecture}
The tropical semiring $(\R \cup \infty, \mathrm{min}, +)$ is an interesting place to do algebra, and it is intimately connected to tropical geometry.  
In this talk, I will introduce tropical polynomials, ideals, congruences, and how the connection with tropical geometry is made via congruences of bend relations.  
Tropical geometry and matroid theory are telling us that we should focus attention on a narrow slice of the world of tropical algebra.  
This leads to the theory of tropical ideals (as developed by Maclagan and Rinc\'{o}n) and an abundance of interesting open questions.  
I will try to summarize what we know and what we don't yet know about tropical ideals and sketch the outlines of the associated notion of tropical schemes that we are beginning to see.

\subsection*{Second lecture}
In 2008 Payne proved that the Berkovich analytification of an affine variety is homeomorphic to the category-theoretic limit of all of its tropicalizations.  We will explore this phenomenon from the perspective of tropical algebra, bend relations, and universal objects in category theory.  We also present our perspective on the linear version of this story.

Tropicalizing a scheme $X$ requires a choice of an embedding into a toric variety.  The limit of all such embeddings exists as a mild generalisation of a toric embedding, and it can be explicitly described.  The tropicalization determined by this embedding has a universal property: it maps to all other tropicalizations.  The Berkovich analytification also has this property, and the two are in fact homeomorphic.

If $X = \mathrm{Spec} \: A$, then the Berkovich analytification is the space of valuations on $A$.  If one admits valuations taking values in idempotent semirings that are not necessarily totally ordered, then the category of valuations on $A$ has an initial object, and the target of this universal valuation is precisely the algebra corresponding to the universal tropicalization.

\subsection*{Third lecture}
In this talk we will set aside geometry and focus on tropicalization via bend relations as a construction in commutative and non-commutative algebra. 
By starting at the level of tensor algebras, constructions such as symmetric algebras, exterior algebras, matrix algebras, and Clifford algebras can be tropicalized.
In the case of exterior algebras, the resulting tropical notion beautifully completes the picture of the Pl\"{u}cker embedding and gives a new perspective on the tropical Pl\"{u}cker relations.
For matrix algebras and Clifford algebras, Morita theory becomes an interesting topic. I will present some facts and some questions.

\section{First lecture}\label{sec:FirstLecture}

\subsection{Three perspectives on varieties}
Let $K$ be a field and $K[X_1,\dots, X_n]$ the ring of polynomials in $n$ variables. For an element $f\in K[X_1,\dots,X_n]$ we write 
\[
    f(X_1,\dots,X_n)=\sum_{\alpha\in\N^n} c_\alpha X^\alpha 
\]
where $c_\alpha=0$ for all but finitely many $\alpha$ and $X^\alpha\coloneqq X_1^{\alpha_1}\cdots X_n^{\alpha_n}$. Here we assume that the set of natural numbers $\N$ contains zero. Let $I \subseteq K[X_1,\dots,X_n]$ be an ideal, and define the algebraic variety
\[
    V(I)\coloneqq\{ x\in K^n \mid f(x)=0 \text{ for all } f\in I \}.
\]
There are three slightly different ways to think about the set $V(I)$. 
\begin{enumerate}[label=(\arabic*)]
    \item\label{DV1} It is the intersection of the zero sets of all polynomials $f\in I$;
    \item\label{DV2} It the solution set to the system of equations $f=0$ for $f\in I$; 
    \item\label{DV3} It is the set of homomorphisms $\Lambda$
\[
    \begin{tikzcd}
        K[X_1,\dots, X_n] \arrow[rr,"\Lambda"]  \arrow[rd,two heads] & & K \\
        & K[X_1,\dots, X_n]/I  \arrow[ru,dashed,swap,"\lambda"] &
    \end{tikzcd}
\]
that factor through the surjection $K[X_1,\dots, X_n] \twoheadrightarrow K[X_1,\dots, X_n]/I$.
\end{enumerate}
Over a field, these three perspectives are clearly mathematically equivalent.  However, when we transition to the world of tropical algebra, we will see that they each offer something slightly different. In perspective \ref{DV1}, we need to use an appropriate notion of zero set, and this is provided by the usual tropical vanishing condition of the min being attained at least twice.  For perspective \ref{DV2}, we argue that consistency with \ref{DV1} requires that the simple equation $f=0$ be replaced with a system of equations that we call the \emph{bend relations}.  In the tropical world, ideals are no longer sufficient to specify quotients, but the system of equations in \ref{DV2} is, and so the quotient appearing in perspective \ref{DV3} should be replaced by the quotient by the bend relations; but then the quotient algebra contains more information that just its set of solutions in the tropical semiring. Shifting focus from the solutions to the quotient leads to the world of tropical schemes as in \cite{Giansiracusa2X_2016, Maclagan_Rincon_2020}.

\subsection{The tropical semiring} 

A semiring $S$ is a ring where an element may not have an additive inverse. More precisely, a \emph{semiring} is a triple $(S,+,\cdot )$, where $S$ is a nonempty set, and $+, \cdot: S\times S \to S$ are operations called \emph{addition} and \emph{multiplication}, respectively. We write $r+s$ for $+(r,s)$ and $rs$ for $\cdot(r,s)$. The addition and multiplication satisfy the following four axioms:
\begin{enumerate}[label=(SR\arabic*)]
    \item\label{axiom: SR1} The pair $(S,+)$ is a commutative monoid with identity $0$;
    \item\label{axiom: SR2} The pair $(S,\cdot)$ is a monoid with identity $1\neq 0$;
    \item\label{axiom: SR3} For all $r,s,t \in S$
    \[
        r(s+t) = rs + rt \quad \text{ and } \quad (r+s)t = rt + st;
    \]
    \item\label{axiom: SR4} For all $s\in S$, one has $0s=s0=0$.
\end{enumerate}
As for rings, given a semiring $(S,+,\cdot)$, to be short with notation, we write $S$ to represent the semiring $(S,+,\cdot)$. 

\begin{Exa}\label{Exa: semirings} As first basic examples we point out that:
    \begin{itemize}
    \item Any ring $(R, +, \cdot)$ is a semiring where every element has an additive inverse; 
    \item The set of natural numbers $\N$ with usual addition and multiplication is an example of a semiring which is not a ring; 
    \item The set $\{0, \infty\}$ with addition given by the minimum, and multiplication by the classical addition is a semiring. 
    \end{itemize}
\end{Exa}

\begin{Exa}[Tropical semiring]\label{Exa: tropical semiring}
    Let $\T\coloneqq\R\cup\{\infty\}$. The \emph{tropical semiring} is the triple $(\T,\oplus,\odot)$, where 
    \[
        \begin{tikzcd}[row sep=0pt,/tikz/column 1/.append style={anchor=base east}
     ,/tikz/column 2/.append style={anchor=base west}]
            \oplus:\T\times\T \arrow[r] & \T \\
            (a,b) \arrow[r, mapsto] & \min\{a,b\}
        \end{tikzcd}
        \quad \text{ and } \quad
        \begin{tikzcd}[row sep=0pt, /tikz/column 1/.append style={anchor=base east}
     ,/tikz/column 2/.append style={anchor=base west}]
            \odot:\T\times\T \arrow[r] & \T \\
            (a,b) \arrow[r, mapsto] & a+b,
        \end{tikzcd}
    \]
    with these operations extended to the element $\infty$ in the intuitively expected way by the rules $a\oplus \infty = a$ (so $\infty$ is
    the identity element for the commutative monoid $(\T, \oplus)$) and $a \odot \infty = \infty $; these operations are
    called \emph{tropical addition} and \emph{tropical multiplication}, respectively.  Note that, for each $a\in \T$, one has $0\odot a = a\odot 0 = a$, so $0$ is the identity element for the {tropical multiplication} $\odot$.  It is not hard to see that this structure satisfies the above axioms: axioms \ref{axiom: SR1}, \ref{axiom: SR2}, and \ref{axiom: SR4} are clear.
    We next verify the distributivity axiom \ref{axiom: SR3}. For any $r,s,t\in\T$, observe that
    \[
        r\odot(s\oplus t) = r+\min\{s,t\} = \min\{r+s,r+t\} = (r\odot s) \oplus (r\odot t),
    \]
    and likewise for distribution of tropical multiplication from the other side.
    Thus the tropical semiring is indeed a semiring.    Furthermore, in $\T$ we have an inverse operation for the tropical multiplication, namely the \emph{tropical division} given by $a\odiv b\coloneqq a - b$ for any $a\in\T$ and $b\in\T^*\coloneqq \R$.
\end{Exa}

A \emph{subsemiring} $R$ of a semiring $(S,+, \cdot)$ is a subset of $S$ such that the restrictions of the addition and multiplication endow the subset $R$ with the structure of a semiring. In other words, a subsemiring of $S$ is a subset $R$ that contains $\{0_S,1_S\}$ such that $R+R=\{r+r'\in S\mid r,r'\in R\}\subseteq R$
and $R\cdot R=\{r\cdot r'\in S\mid r,r'\in R\}\subseteq R$.

Let $S$ and $S'$ be semirings, a \emph{morphism} of semirings $\varphi:S\to S'$ is a map that preserves the identity elements, sums, and multiplications, that is, $\varphi(0_S)=0_{S'}$, $\varphi(1_S)=1_{S'}$, $\varphi(r+s)=\varphi(r)+\varphi(s)$, and $\varphi(rs)=\varphi(r)\varphi(s)$. When the morphism $\varphi$ is a bijection we say it is an \emph{isomorphism} between $S$ and $S'$ and write $S\cong S'$. For more details about semirings, see \cite{Golan1999} and \cite{Golan2003}. 

\begin{Exa}
    Note that $\N$ is a subsemiring of $\Q_{\ge0}$ and $\R_{\ge0}$, both of them with usual addition and multiplication. More generally, $\N$ is the initial object in the category of semirings, which is to say that any semiring $S$ admits a unique homomorphism $\N \to S$; the image of $\N$ in $S$ is the minimal subsemiring of $S$, consisting of the elements $\{0,\:1_{S},\:1_{S}+1_{S}, \:1_{S}+1_{S}+1_{S}, \: \ldots\}$.
\end{Exa}
\begin{Exa}[Models for the tropical semiring]
Let $\T_{\max}=(\R\cup\{-\infty\},\oplus_m,\odot_m)$ where 
$a\oplus_m b\coloneqq\max\{a,b\}$ and $a\odot_m b\coloneqq a\odot b=a+b$. It is not difficult to verify that $\T_{\max}$ with this addition  and multiplication is a semiring with $0_{\T_{\max}}=-\infty$ and $1_{\T_{\max}}=0$. Now, we define a map from the tropical semiring $\T$ to $\T_{\max}$ by 
\[
    \begin{tikzcd}[row sep=0pt,/tikz/column 1/.append style={anchor=base east}
     ,/tikz/column 2/.append style={anchor=base west}]
            \varphi:\T\arrow[r] & \T_{\max} \\
            t \arrow[r, mapsto] & -t \\
            \infty \arrow[r, mapsto] & -\infty.
        \end{tikzcd}
\]
This is a bijection with inverse given by $t \mapsto -t$. Moreover, $\varphi(0)=0$ and $\varphi(a\odot b)=-a-b=\varphi(a)\odot_m\varphi(b)$. Finally, 
\begin{align*}
    \varphi(a\oplus b) & = \varphi(\min\{a,b\}) = -\min\{a,b\} =\max\{-a,-b\} = \max\{\varphi(a),\varphi(b)\} \\
    & = \varphi(a)\oplus_m \varphi(b).
\end{align*}
Thus we have an isomorphism of semirings $\T \cong \T_{\max}$.

There is a third model for the tropical semifield. Namely, consider the triple $\T_e=(\R_{\ge0},\oplus_m,\cdot)$ where $a \oplus_m b\coloneqq\max\{a,b\}$ and $a \odot b\coloneqq ab$ (usual multiplication). It is not difficult to verify that $\T_e$ is a semifield with neutral elements $0_{\T_e}=0$ and $1_{\T_e}=1$ for addition and multiplication, respectively. Moreover, the map $\exp \co \T\to\T_{e}$ given by $t\mapsto \exp(-t)$ and $\infty \mapsto 0$ gives an isomorphism $\T\cong \T_e$.
\end{Exa}

\begin{Exa}[$\T$ as a limit of semirings, {\cite[$\S$~1.2]{Shaw_2015}}]\label{Exa: T as a limit of semirings}
    Here we fix the $\T_{\max}$ model for the tropical semiring. Let $\ell>1$ be a real number, and define the semiring $\T_\ell\coloneqq(\R\cup\{-\infty\}, \oplus_\ell,\odot_\ell)$, where
    \[
        \begin{tikzcd}[row sep=0pt,/tikz/column 1/.append style={anchor=base east},/tikz/column 2/.append style={anchor=base west}]
            \oplus_\ell:\T_\ell\times\T_\ell \arrow[r] & \T_\ell \\
            (a,b) \arrow[r, mapsto] & \log_\ell(\ell^a+\ell^b)
        \end{tikzcd}
        \quad \text{ and } \quad
        \begin{tikzcd}[row sep=0pt, /tikz/column 1/.append style={anchor=base east},/tikz/column 2/.append style={anchor=base west}]
            \odot_\ell:\T_\ell\times\T_\ell \arrow[r] & \T_\ell \\
            (a,b) \arrow[r, mapsto] & \log_\ell(\ell^a \ell^b).
        \end{tikzcd}
    \]
    For each $\ell>1$ the identity map gives a bijection of sets between $\T$ and $\T_{\ell}$; this is \emph{not} a semiring isomorphism; 
    the identity map preserves $0_{\T_\ell}=-\infty$, $1_{\T_\ell}=0$, and multiplication, but it does not preserve addition.  However, it is approximately an isomorphism in the following quantitative sense: applying  $\log_\ell$ in the inequalities $\ell^a\oplus_m \ell^b \le \ell^a + \ell^b \le 2\odot(\ell^a\oplus_m\ell^b)$, one obtains 
    \[
        a\oplus_m b \le a \oplus_\ell b \le a\oplus_m b + \log_\ell(2)
    \]
     for each $\ell > 1$. In this sense the tropical semiring $\T_{\max}$ can be seen as the limit $\ell\to\infty$ of the semirings $\T_\ell$ 
     (see \autoref{fig:family_Trop_semiring}).
    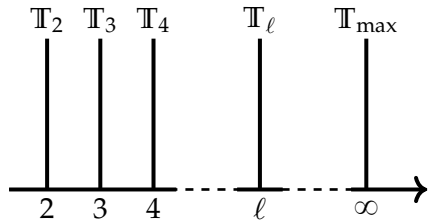
\begin{figure}[ht]
    \centering
    \begin{tikzpicture}   
        \draw[line width = 1.5pt, -] (0,0) -- (0,2) node[yshift=7pt, xshift=0pt] {$\T_2$};
        \node[yshift=-7pt] at (0,0) {$2$};
        \draw[line width = 1.5pt, -] (0.7,0) -- (0.7,2) node[yshift=7pt, xshift=0pt] {$\T_3$};
        \node[yshift=-7pt] at (0.7,0) {$3$};
        \draw[line width = 1.5pt, -, name path = 4] (1.4,0) -- (1.4,2) node[yshift=7pt, xshift=0pt] {$\T_4$};
        \node[yshift=-7pt] at (1.4,0) {$4$};
        \draw[line width = 1.5pt, -] (2.8,0) -- (2.8,2) node[yshift=7pt, xshift=0pt] {$\T_\ell$};
        \node[yshift=-7pt] at (2.8,0) {$\ell$};
        \draw[line width = 1.5pt, -, name path = infty] (4.2,0) -- (4.2,2) node[yshift=7pt, xshift=0pt] {$\T_{\max}$};
        \node[yshift=-7pt] at (4.2,0) {$\infty$};
        \draw[line width = 1.5pt, -, name path=ell] (-.5,0) -- (1.7,0) node[xshift=6pt] {};
        \draw[line width = 1.5pt, -, name path=ell] (2.5,0) -- (3.1,0) node[xshift=6pt] {};
        \draw[dashed, line width = 1pt,name path=deg_d] (-.5,0) -- (5,0) node[xshift=-23pt] {};
        \draw[line width = 1.5pt, ->, name path=ell] (4,0) -- (5,0) node[xshift=6pt] {};
        \draw[dashed, line width = 1pt,name path=deg_d] (4,0) -- (5,0) node[xshift=-23pt] {};
    \end{tikzpicture}
    \caption{Family of semirings converging to $\T$.}
    \label{fig:family_Trop_semiring}
\end{figure}
The idea of considering this limit of semirings is due to Litvinov and Maslov. This approach is known as \emph{Maslov dequantization}; see, for instance, \cite{Viro_2000}.
\end{Exa}
Let $S$ and  $S'$ be two semirings. Their product $S\times S'$, with componentwise addition and multiplication, is naturally a semiring with neutral elements for addition and multiplication $(0_{S}, 0_{S'})$ and $(1_S, 1_{S'})$, respectively.

\subsection{Quotients of semirings}

A \emph{congruence} $\sim$ on a semiring $S$ is an equivalence relation on $S$ such that the semiring operations of $S$ descend to well-defined operations on the quotient $S/\hspace{-3pt}\sim$. It is straightforward to show that an equivalence relation is a congruence if and only if it satisfies the following two conditions (for any $a,b,r,s \in S$):
\begin{enumerate}[label=(CS\arabic*)]
    \item\label{axiom: SRC1} If $a\sim b$ and $r\sim s$, then $a+r \sim b+s$;
    \item\label{axiom: SRC2} If $a\sim b$ and $r\sim s$, then $ar \sim bs$. 
\end{enumerate}
In other words, $\sim$ is a congruence on $S$ if and only if 
\[
    \{ (r,s) \in S\times S \mid r \sim s \}
\]
is a subsemiring of $S\times S$. A congruence $\sim$ on $S$ is \emph{proper} if there exist $a,b\in S$ with $a\nsim b$, that is, if $\{ (r,s) \in S\times S \mid r \sim s \}$ is a proper subsemiring of $S\times S$. As for many algebraic structures, the quotient $S/\hspace{-3pt}\sim$ by a proper congruence is naturally endowed with the structure of a semiring. Namely, let $[s]$ be the class of $s\in S$ in the quotient $S/\hspace{-3pt}\sim$. The addition and multiplication in $S/\hspace{-3pt}\sim$ are defined by
\[
     [r]+[s]\coloneqq[r+s] \quad \text{ and } \quad [r][s]\coloneqq[rs],
\]
respectively. When the congruence is not proper, then the quotient is a single point and we have violated the part of semiring axiom \ref{axiom: SR2} requiring $0\neq 1$.

Given a morphism $\varphi: S\to S'$ of semirings, we define its \emph{kernel congruence} by 
\[
    \ker(\varphi)\coloneqq \{ (r,s) \in S\times S \mid \varphi(r)=\varphi(s)\}.
\]
The image of $\varphi$, defined as $\im(\varphi)\coloneqq\{\varphi(s)\mid s\in S\}$, is a subsemiring of $S'$. Replacing ideals by congruences, analogues of the classical isomorphism theorems for rings hold for semirings; see \cite[Prop. 2.4.4]{Giansiracusa2X_2016} for details.

The \emph{congruence generated} by a set of pairs $\{(r_i, s_i)\}_{i\in \Lambda}$, denoted by $\langle r_i \sim s_i \rangle_{i\in \Lambda}$, is the smallest congruence such that $r_i\sim s_i$ for all $i\in \Lambda$. 
It can be described as the intersection of all congruences containing the relations $r_i \sim s_i$; this uses the easy-to-check fact that an intersection of congruences is again a congruence.
Alternatively, the congruence generated by a set of pairs can be constructed by first taking the subsemiring of $S\times S$ generated by the pairs $(r_i, s_i)$, and then taking the symmetrisation and transitive closure by adding $(a,c)$ whenever $(a,b)$ and $(b,c)$ are present.

One might wonder why we need to talk about congruences when we all learn about quotients in algebra in terms of ideals.  The answer is that ideals and congruences are entirely equivalent descriptions of quotients of rings, but for semirings there are many more congruences than ideals, and so there are quotients that are not specified by ideals.  We now examine this in more detail.

An ideal $I \subset S$ determines a congruence generated by the relations $\{f\sim 0\}_{f\in I}$.  This gives a map 
\[
\{\text{Ideals in $S$}\} \stackrel{\mathcal{I}}{\longrightarrow} \{\text{Congruences on $S$}\}.
\]
Going in the opposite direction, a congruence $\mathcal{C}$ determines an ideal $\{f \:\: | f\sim_{\mathcal{C}} 0\}$, and so there is a map
\[
\{\text{Congruences on $S$}\} \stackrel{\mathcal{K}}{\longrightarrow} \{\text{Ideals in $S$}\}.
\]
When $S$ is a ring, any relation $r\sim s$ implies and is implied by the relation $r-s \sim 0$ since we can subtract or add $s$ on both sides.  Hence the above pairs of maps are bijections inverse to one another. In contrast, on a semiring where we do not have additive inverses, the correspondence above can fail. The composition $\mathcal{K} \circ \mathcal{I}$ is always the identity (so $\mathcal{I}$ is injective and $\mathcal{K}$ is surjective), but the reverse composition need not be the identity.

\begin{Exa}
On $\T[x,y]$, consider the congruence generated by the relation $x\sim y$. The quotient map $\T[x,y] \to \T[x,y]/(x\sim y)$ does not identify any elements with $\infty = 0_\T$, Hence $\mathcal{K}$ sends both this congruence and the trivial congruence to the trivial ideal $(0_\T)$.
\end{Exa}

For further details concerning congruences on semirings, we refer the reader to \cite{Giansiracusa2X_2016}, \cite{Giansiracusa2x2018}, and \cite{Golan2003}.

\subsection{Tropical polynomials}
Let $\T[X_1,\dots,X_n]$ be the semiring of tropical polynomials in $n$ variables, and $f\in \T[X_1,\dots,X_n]$. We write 
\[
    f(X_1,\dots,X_n)=\bigoplus_{\alpha\in\N^n} c_\alpha \odot X^\alpha
\]
where $c_\alpha = \infty$ for all but finitely many $\alpha\in\N^n$, and $X^\alpha$ denotes the tropical monomial $X_1^{\alpha_1}\odot\cdots\odot X_n^{\alpha_n}$.  
Note that, in analogy with the classical case, where $1X^\alpha$ is written only as $X^\alpha$, in the tropical notation $X^\alpha=0\odot X^\alpha$.  A polynomial $f$ is an algebraic expression, and it represents a function $\T^n \to \T$ given by sending 
\[
(x_1, \ldots, x_n) \mapsto f(x_1, \ldots, x_n) = \min_{\alpha\in\N^n} \left \{ c_\alpha + \langle  \alpha, x\rangle \right \}
\]
i.e., we are evaluating $f$ at the point $(x_1, \ldots, x_n)$ using the arithmetic operations of $\T$.

Analogously to the classical case, the \emph{degree} of the tropical polynomial  $f$  is defined as $\deg(f)\coloneqq\max\{\alpha_1+\cdots+\alpha_n \mid c_\alpha\neq\infty\}$. We say that $f$ is \emph{homogeneous} (of degree $d$) if each monomial $c_\alpha\odot X^\alpha$ of $f$ with $c_\alpha\neq\infty$ has degree $d$. 
For any pair $f$ and $g$ of tropical polynomials, we clearly have $\deg(f\odot g) = \deg(f)\odot\deg(g)$. Thus, as in the world of rings, the semiring of tropical polynomials is graded by degree, which is to say that $\T[X_1,\dots,X_n]$ can be written as the (direct) sum 
\[
\oplus_{d \in \N} \T[X_1,\dots,X_n]_{d}
\]
where the $d$ piece consists of the set of all homogeneous tropical polynomials of degree $d$. Note that the $\T$-module $\T[X_1,\dots,X_n]_{d}$ is free of rank $\binom{n+d-1}{d}$.

\subsection{The vanishing locus of a tropical polynomial}

We are now ready to revisit the three perspectives on tropical varieties from the beginning of this lecture.

First we consider the tropical world's version of the zero set of a polynomial.   The function 
$\T^n \to \T$ defined by a tropical polynomial attains the value $\infty = 0_\T$ either nowhere or only at the point $(\infty, \ldots, \infty)$ if it has no constant term; moreover, a tropical Laurent polynomial defines a function $(\T\smallsetminus \{0\})^n = \R^n \to \T$ that never attains the value $\infty$.  Thus the preimage of $\infty$ is not the appropriate definition of the variety it determines.

Instead, the widely accepted definition of the variety of a tropical polynomial is given by the set of points satisfying usual min-attained-at-least-twice condition.  The following argument for the appropriateness of this definition is due to Mikhalkin.  Given $f\in \T[X_1^\pm, \dots, X_n^\pm]$, we consider its associated function $f\co \R^n\to \R$. Its graph has the following three properties: 
\begin{enumerate}
\item It is piecewise linear;
\item It has slopes in $\Z$;
\item It is convex downward.
\end{enumerate}
Functions satisfying these three conditions are said to be \emph{(tropically) regular}. 
\autoref{fig:graph_f} provides an illustrative example.
    
    \begin{figure}[ht]
        \centering
        \resizebox{4cm}{3cm}{%
        \begin{tikzpicture} 
            \draw[->] (0,-.5) -- (0,3) node[yshift=5pt] {};
            \draw[->] (-.5,0) -- (5,0) node[xshift=6pt] {};

            \draw[line width = 1pt,name path=f2] (.4,2) -- (3.6,2) node[xshift=5pt] {$2$}; 
            \draw[line width = 1pt,,name path=2x] (-.5,-1) -- (1.5,3) node[yshift=-5pt, xshift=-15pt] {$2X$}; 
            \draw[line width = 1pt,name path=-2x+8] (2.5,3) -- (4.5,-1) node[yshift=110pt, xshift=-30pt] {$8-2X$}; 
            \draw[name path=-2x+8] (2.5,3) -- (4.5,-1) node[red, yshift=73pt, xshift=-74pt] {$f(X)$}; 
            \draw[line width = 1.5pt, red,name path=graph_f] (-.5,-1) -- (1,2) -- (3,2) -- (4.5,-1);
        \end{tikzpicture}
        }
        \caption{The graph of the tropical Laurent polynomial $f(X)= (X^2)\oplus (2)\oplus (8\odot X^{-2})$ is shown in red.  It is the minimum of the linear functions represented by its three monomial terms.}
        \label{fig:graph_f}
    \end{figure}
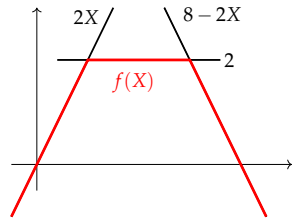

A root of a classical polynomial (or Laurent polynomial) $P$ is a point where $1/P$ is not locally regular in the classical sense. Now consider a tropical Laurent polynomial $f$. 
The tropical reciprocal $f^{-1} = 0 \odiv f$ has graph equal to the graph of $-f$ (i.e., the reflection across the $x$-axis). A \emph{tropical root} is a point in the domain where $f^{-1}$ is not tropically regular on a neighbourhood. Since taking the reciprocal in the tropical sense simply flips the graph upside down, it preserves piecewise linearity and $\Z$-slopes.  However, while $f$ is convex downward, $f^{-1}$ is convex upward.  Locally near a point where $f$ is linear, the graph is both convex upward and downward, but at a point of where the graph is nonlinear, $f^{-1}$ fails to be locally convex downward, and hence it fails to be regular. The graph of $f$ is the minimum of a set of affine linear functions, one for each monomial term in $f$, and these points of non-linearity are the points in the domain where this minimum of affine linear terms is attained by at least two terms. Thus one is led to define 
 \emph{tropical zero locus} or \emph{bend locus} of $f$ as the set of points where the minimum is attained at least twice.  That is,   
\[
    V^{\trop} (f)  \coloneqq \left\{ x\in \R^n\mid\min_{\alpha}\{c_\alpha + \langle \alpha, x\rangle \} \text{ is attained at least twice}\right\}.
\]
See an example in \autoref{fig:breaks_at_the_band_point} for a simple illustration.

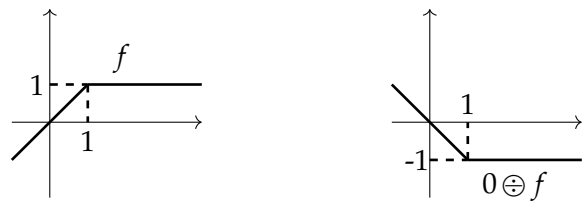
\begin{figure}[ht]
    \centering
    \begin{tikzpicture}
        \def\mx{.5} 
        \pgfmathsetmacro{\my}{.5} 
        \pgfmathsetmacro{\myn}{-\my} 
        \pgfmathsetmacro{\tx}{5} 
        \pgfmathsetmacro{\lx}{-.5} 
        \draw[->] (0,-1) -- (0,1.5) node[] {};
        \draw[->] (-.5,0) -- (2,0) node[] {};

        \draw[line width = 1pt] (\lx,\lx) -- (\mx,\my);
        \draw[line width = 1pt, dashed] (0,\my) -- (\mx,\my) node[xshift=-19pt] {1}; 
        \draw[line width = 1pt] (\mx,\my) -- (2,\my) node[xshift=-30pt, yshift=10pt] {$f$};
        \draw[line width = 1pt, dashed] (\mx,\my) -- (\mx,0) node[yshift=-7pt] {1};
        
        \draw[->] (\tx + 0,-1) -- (\tx + 0,1.5) node[] {};
        \draw[->] (\tx -.5,0) -- (\tx + 2,0) node[] {};

        \draw[line width = 1pt] (\tx + \lx, -2*\myn + \lx ) -- (\tx + \mx, \myn);
        \draw[line width = 1pt, dashed] (\tx + 0,\myn) -- (\tx + \mx,\myn) node[xshift=-19pt] {-1}; 
        \draw[line width = 1pt] (\tx + \mx,\myn) -- (\tx + 2,\myn) node[xshift=-25pt, yshift=-10pt] {$0\odiv f$};
        \draw[line width = 1pt, dashed] (\tx+ \mx,\myn) -- (\tx + \mx,0) node[yshift=7pt] {1};
    \end{tikzpicture}
    \caption{The graphs of $f=1 \oplus x$ and $f^{-1}$.}
    \label{fig:breaks_at_the_band_point}
\end{figure}
 
An illustration of $V^{\trop}$, is shown in the next Example:
\begin{Exa}
    Consider $f(X,Y)= X\oplus Y \oplus 0 \in \T[X^\pm, Y^\pm]$. By definition, $V^\Vtrop(f)$ is the locus where $\min\{X,Y,0\}$ is attained at least twice. Each of these three terms dominates in a region (0 in the upper right quadrant, $Y$ in the lower region, and $X$ in the left region), as shown in \autoref{fig:trop_hhypersurface_1}.  The variety $V^\Vtrop(f)$ is the black rays forming a Y-shaped graph where the regions meet.
    \begin{figure}[ht]
        \centering
        \begin{tikzpicture}
            \node (y0x) [] at ($(2.5,0)$) {$Y=0\leq X$};
            \node (x0y) [] at ($(0,1.7)$) {$X=0\leq Y$};
            \node (xy0) [] at ($(-1.2,-1.3)$) {$Y=X\leq 0$};
            \node[black!80!green!100] (0) [] at ($(.65,.65)$) {$0$};
            \node[blue] (x) [] at ($(-.5,.2)$) {$X$};
            \node[red] (y) [] at ($(.2,-.5)$) {$Y$};
            
            \path[draw, line width=1.5pt] (0,0) -- (-1,-1); 
            \path[draw, line width=1.5pt] (0,0) -- (0,1.4);
            \path[draw, line width=1.5pt] (0,0) -- (1.4,0);

            \fill[color=red!90,opacity=.30] 
            (0,0) -- (-1,-1) -- (.4,-1) -- (1.4,0);
            \fill[color=blue!90,opacity=.30] 
            (0,0) -- (-1,-1) -- (-1,.4) -- (0,1.4);
            \fill[color=green!90,opacity=.30] 
            (0,0) -- (1.4,0) -- (1.4,1.4) -- (0,1.4);
            
        \end{tikzpicture}
        \caption{Ilustration of $V^\Vtrop(f)$.}
        \label{fig:trop_hhypersurface_1}
    \end{figure}
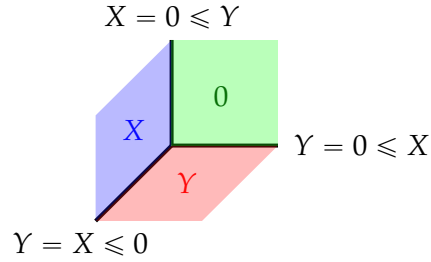
\end{Exa}
The bend locus of a single tropical Laurent polynomial is called a \emph{tropical hypersurface}.

\begin{remark}
A general \emph{tropical variety} is the intersection of a collection of tropical hypersurfaces satisfying some conditions that we will not go into here.  These conditions are satisfied when the collection consists of all tropicalizations of elements in an ideal, and they guarantee that the resulting set satisfies the balancing condition.  
\end{remark}

\subsection{Tropical varieties are solutions to equations}
We are now ready to come to the fundamental question of this lecture. We saw that, in passing from the classical to the tropical world, the definition of a variety as zero locus should be replaced with the definition as tropical vanishing locus.  Our second perspective for classical varieties was as the solution set to equations $f=0$. 
\begin{center}
    \emph{Given a tropical polynomial $f$, what is the system of equations whose solution set is the tropical hypersurface $V^\Vtrop(f)$?}
\end{center}
Note that a system of equations in a semiring $S$ is the same thing as a congruence on $S$.

The answer to the above question is not unique! This should not be surprising to anyone who has studied scheme theory and is familiar with non-reduced structures; distinct schemes over a field $K$ can have the same $K$-valued points.  For instance, the affine schemes $\Spec(K[X]/\langle X \rangle)$ and $\Spec(K[X]/\langle X^2 \rangle)$ both have a single $K$ point but they are most definitely not isomorphic as their structure sheaves are different.

There is a \textbf{maximal choice}.  Given a set $X\subseteq \R^n$, we can consider the set of equations $f = g$ for all pairs of tropical polynomials satisfying $f(x)=g(x)$ for all $x \in X$. When $X$ is a sufficiently nice subset (a large class that includes polyhedral complexes such as tropical varieties), the solution set for this set of equations recovers $X$.  This maximal choice loses information and corresponds to taking the reduced scheme structure, which in the affine case means passing from an ideal to its radical.

There are many other possible choices, but we would like one that depends in some canonical way on the tropical polynomials $f$.  

Let $f=\bigoplus_{\alpha\in\Z^n}c_\alpha\odot X^\alpha$ be a Laurent polynomial  in $\T[X_1^\pm,\dots,X_n^\pm]$. The \emph{support} of $f$, denoted by $\Supp(f)$, is the set of $\alpha$ such that $c_\alpha\neq\infty$. That is, 
\[
    \Supp(f)\coloneqq\{ \alpha \in \Z^n \mid c_\alpha \neq \infty \}.
\]
Given an exponent vector $\beta\in \Z^n$, we define $f_{\widehat{\beta}}$ as the (Laurent) polynomial obtained by deleting the $\beta$-monomial $c_\beta\odot X^\beta$ from $f$.  In symbols, 
\[
    f_{\widehat{\beta}}\coloneqq \bigoplus_{\substack{ \alpha\in\Z^n \\ \alpha\neq\beta }}c_\alpha\odot X^\alpha.
\]
Observe that if $\beta\notin\Supp(f)$, then $f_{\widehat{\beta}} = f$.

\begin{definition} \label{def:bend-congruence}
    Let $f=\bigoplus_{\alpha}c_\alpha\odot X^\alpha$ be a Laurent polynomial.
The \emph{bend relations} of $f$ are the equations
\[
    \left\{ f \sim f_{\widehat{\beta}} \mid \beta \in \Supp(f)\right\},
\]
and the \emph{bend congruence} of $f$ denoted $\Bb(f)$, is the congruence generated by the bend relations of $f$.
\end{definition}

\begin{Exa}
    Let $f(X,Y) = X\oplus Y\oplus 0$. The bend relations of $f$ are
    \[
        \{X\oplus Y\oplus 0 \sim X\oplus Y,\; X\oplus Y\oplus 0 \sim X \oplus 0,\; X\oplus Y\oplus 0 \sim Y \oplus 0\}.
    \]
            

    
\end{Exa}

\begin{proposition}
The solution set to the bend relations of a tropical polynomial (or tropical Laurent polynomial) $f$ is precisely the tropical hypersurface $V^\Vtrop(f)$.
\end{proposition}
\begin{proof}
Consider a point $p \in \T^n$.  The value of $f$ at $p$ is the minimum of the value of each linear term at $p$.  The valued of $f_{\widehat{\beta}}$ at $p$ is the minimum of all the terms other than the $\beta$ term.  These two are equal if and only if the minimum is not attained solely by the $\beta$ term.  We thus have the equality $f(p) = f_{\widehat{\beta}}(p)$ for all $\beta$ if and only if the minimum is not attained solely by any single term, which is to say that $f$ tropically vanishes at $p$.  
\end{proof}

The above proposition begins to justify our assertion that the bend relations of $f$ are the appropriate answer to the question of what equations define a tropical variety. 

Adapting our third perspective on  varieties to the tropical setting now proceeds easily: the set of solutions to the bend relations of $f$ is precisely the set of $\T$-algebra homomorphisms $\T[X_1^\pm, \ldots, X_n^\pm] \to \T$ that descend to the quotient by the congruence $\Bb(f)$.

Note that  $\Bb(f)$ determines $f$ up to multiplication by a scalar. 
    Indeed, let $f$ be a homogeneous tropical polynomial of degree $d$, that is $f$ is in $\T[X_1,\dots,X_n]_d \cong \T^N$, where
    $N=\binom{n+d-1}{d}$. Now, note that 
    \[
        \Hom_{\T\text{-mod}}\hspace{-3pt}\Big(\hspace{-2pt}\big(\T[X_1,\dots,X_n]/\Bb(f)\big)_d, \T\Big) = \big\{w\in \T^N \mid w\perp f\big\} = f^\perp,
    \]
    where $\langle w, f \rangle\coloneqq\bigoplus_{i=1}^{N} w_i\odot f_i$, and by definition
    \[
        f^\perp\coloneqq\big\{w\in \T^N\mid \min_{i}\{w_i+f_i\} \text{ is attained at least twice} \big\}.
    \]
    Finally, a standard fact from tropical linear algebra gives us that  $(f^\perp)^\perp = \Span(f)$.
\subsection{Tropicalization} There are many different constructions called tropicalization to be found in the literature.  In these lectures we will focus only on the following ones:
\begin{enumerate}
    \item Tropicalizing polynomials by valuating their coefficients;
    \item Tropicalizing varieties by either valuating the coordinates of their points or by tropicalizing their defining polynomials;
    \item Tropicalizing linear spaces as a special case of varieties.
\end{enumerate}

A \emph{valuation} on a field $K$ is a map $\nu\co K\to \T$ such that, for all $a,b\in K$ it satisfies the following three axioms:
\begin{enumerate}[label = (V\arabic*)]
    \item\label{axiom: valuatioun1} $\nu^{-1}(\infty)=\{0\}$;
    \item\label{axiom: valuatioun2} $\nu(ab) = \nu(a)\odot \nu(b)$;
    \item\label{axiom: valuatioun3} $\nu(a+b)\ge \nu(a)\oplus\nu(b)$.
\end{enumerate}
Let $\nu$ be a valuation. By Axiom \ref{axiom: valuatioun2}, it follows that $\nu(a)=\nu(1a)$ is equal to $\nu(1)\odot\nu(a)=\nu(1)+\nu(a)$ for all $a\in K$. This implies that $\nu(1)=1_\T=0$. Moreover, $0=\nu(1)=\nu\big((-1)^2\big)=2\nu(-1)$, i.e.\ $\nu(-1)=1_\T=0$. Therefore $\nu(-a)=\nu(a)$ for each $a\in K$. For an element $a\in K^*$, we have $0=\nu(aa^{-1})=\nu(a)+\nu(a^{-1})$, i.e.\ $\nu(a^{-1})=-\nu(a)$.  
Furthermore, if $\nu$ is a valuation, then for each $\lambda\in \R_{>0}$ the map $\lambda \nu$ (ordinary multiplication) is also a valuation. The \emph{trivial valuation} is the map $\mu\co K\rightarrow \T$ defined as $\mu(t)=1_\T = 0$ if $t\in K^*$ and $\mu(0)=0_{\T} = \infty$. For the readers interested in learning more about valuations, we refer to \cite{Ribenboim_1999}.

\begin{Exa}[$p$-adic valuation]\label{Exa: valuation p-adic}
    Let $p\in \N$ be prime. The \emph{$p$-adic valuation} on the rational numbers $\Q$ is the valuation $\nu_p\co\Q\to\T$ defined by:
    \[
        \nu_p(r)\coloneqq
        \begin{cases}
            \ell, & \text{if } r=p^\ell \displaystyle\frac{a}{b} \quad \text{ and } \quad p\nmid ab, \\[-11pt] \\
            \infty, & \textup{ if } r=0,
        \end{cases}
    \]
    where $a$ and $b$ are integers. For example, if $p=3$, then $\nu_{3}(5/6)=-1$ and $\nu_3(27/7)=3$.
\end{Exa}

\begin{Exa}[Puiseux series]\label{Exa: valuation Puiseux series}
    Let $K\{\hspace{-3pt}\{t\}\hspace{-3pt}\} $ be the field of \emph{Puiseux series} with coefficients in the field $K$. An element $s(t)$ of $K\{\hspace{-3pt}\{t\}\hspace{-3pt}\}$ is a formal power series
    \[
        s(t)= s_1t^{k_1/n}+s_2t^{k_2/n}+s_3t^{k_3/n}+\cdots,
    \]
    where each $s_i\in K$, $n$ is a non-null natural number, and $k_1<k_2<\cdots$ are integers. We define a valuation $\tau\co K\{\hspace{-3pt}\{t\}\hspace{-3pt}\}\to \T$ by $\tau\big(s(t)\big)=\min\{k_i/n\mid s_i\neq 0\}$ if $s(t)$ is in $K\{\hspace{-3pt}\{t\}\hspace{-3pt}\}^*$ and $\tau(0)=\infty$.
\end{Exa}

Let
$\nu: K \rightarrow \T$ be a valuation, and $F = \sum_{\alpha}c_\alpha X^\alpha$ a polynomial in $K[X_1,\dots,X_n]$. We apply the valuation $\nu$ on the coefficients of $F$ to obtain a tropical polynomial $\Vtrop(F)$ in $\T[X_1,\dots,X_n]$, i.e.\ 
\[
    \Vtrop(F)\coloneqq\bigoplus_{\alpha}\nu(c_\alpha)\odot X^\alpha.   
\]
If $J\subseteq K[X_1,\dots,X_n]$ is an ideal, then its \emph{tropicalization} is defined as
\[
    \Vtrop(J)\coloneqq \langle \Vtrop(g) \mid g\in J \rangle \subseteq \T[X_1,\dots,X_n].   
\]
If the ideal $J$ is homogeneous, then $\Vtrop(J)$ is a homogeneous ideal as well. Note that, to obtain the tropicalization of an ideal, we need to tropicalize all polynomials $g\in J$, not just a generating set. For instance, if $J$ is finitely generated, say $J=\langle g_1, \dots, g_\ell\rangle$, then $\langle \Vtrop(g_1),\dots,\Vtrop(g_\ell)\rangle \subseteq \Vtrop(J)$, but a proper inclusion can occur. 
\begin{Exa}\label{Exa: trop_ideal}
Consider $g=X+1 \in K[X]$.  Then $\Vtrop(g) = X\oplus 0$.  Since cancellations cannot happen over $\T$, given any $h \in \langle \Vtrop(g) \rangle$, if the leading monomial term of $h$ is proportional to $X^n$ then $X^{n-1}$ must also be present.  However, $(X-1)g = X^2 - 1$, and so the tropicalization of the ideal $\langle g \rangle$ contains $X^2 \oplus 0$, which doesn't satisfy the above condition and hence is not in $\langle \Vtrop(g) \rangle$.
\end{Exa}

Consider an ideal $J\subseteq K[X_1,\dots,X_n]$ over an algebraically closed field $K$ equipped with a valuation.  For simplicity, let us assume that the valuation is surjective to $\T$ (this can always be arranged by passing to a valued field extension).  The ideal $J$ defines a variety $V(J) \subset K^n$.  The tropicalization of the set $V(J)$ can be defined as 
\[
    \Vtrop\big(V(J)\big) \coloneqq \left\{\big(v(p_1), \ldots, v(p_n)\big) \in T^n \:\: | \:\: (p_1, \ldots, p_n) \in V(J)\right\}.
\]
However, thanks to the Fundamental Theorem of Tropical Geometry, this tropicalization can alternatively be described as the intersection of all tropical hypersurfaces associated with the tropicalizations of polynomials in $J$:
\[
   \Vtrop\big(V(J)\big)  = \bigcap_{f\in\Vtrop(J)} V^\Vtrop(f) \subseteq \T^n.
\]

We have seen that a tropical hypersurface $V^\Vtrop(f)$ can be described as the solution set to the system of equations given by the bend relations of $f$, and so $\Vtrop\big(V(J)\big)$ is the solution set to the bend relations of $\Vtrop(J)$.  Following Grothendieck's insight in developing scheme theory, a system of equations contains more information than its solution set over a particular coefficient (semi)field, and so we should consider the systems of equations as geometric objects in their own right.  This motivates a definition:

\begin{definition}  The \emph{scheme-theoretic tropicalization} of an affine scheme $\Spec K[X_1, \ldots, X_n]/J$ embedded in $\mathbb{A}_K^n$ is the object whose algebra of regular functions is
\[
\T[X_1, \ldots, X_n]/\Bb\big(\Vtrop(J)\big).
\]
\end{definition}

Ideals in $\T[X_1 \ldots, X_n]$ of the form $\trop(J)$ have some nice properties that can be abstracted to the notion of \emph{tropical ideals}. To explain this, we first need a short digression to recall some basic facts about tropical linear spaces.

\subsection{Tropical linear spaces}
Let $E$ be a (finite) set, and $r\in\N$. We set $\binom{E}{r}\coloneqq\{V\in 2^E \mid \#V=r\}$. A \emph{valuated matroid} on the ground set $E$ and rank $r$ is a pair $M=(E,\mu)$ where $\mu:\binom{E}{r}\to \T$ is a non-trivial map, i.e.\ $\mu(B)\neq\infty$ for some $B\in\binom{E}{r}$, satisfying the following axiom:  
    \begin{enumerate}[label=(VM)]
        \item\label{axiom: valuated_matroid}
            For every $B,C \in \binom{E}{r}$, and $b\in B\setminus C$ there exist $c\in C\setminus B$ such that 
         \[
            \mu(B) + \mu(C)\ge \mu\big((B\cup\{c\})\setminus\{b\}\big) + \mu\big((C\cup\{b\})\setminus\{c\}\big).
        \]
     \end{enumerate}
Dress and Wenzel were the first to define and work with valuated matroids in their seminal work \cite{Dress_Wenzel_1992}. 

The \emph{valuated basis exchange axiom} \ref{axiom: valuated_matroid} is equivalent to the \emph{Pl\"{u}cker relations}; see, for instance \cite{Baker_Bowler_2019} and \cite{Jarra_Lorscheid_Vital_2024}.
A set $B\in\binom{E}{r}$ is called a \emph{basis} of the valuated matroid $M$ if $\mu(B)\neq 0_{\T}$. Let $N=(E,\mathcal{B})$ be a (classical) matroid with bases $\mathcal{B}$ and rank $r$. Define $\mu_N: \binom{E}{r}\to \T$ as $\mu_N(B)=1_\T$ if $B\in\mathcal{B}$, and $\mu_N(B)=0_\T$ otherwise. Then, $M_N=(E,\mu)$ is a valuated matroid.
In this sense, a valuated matroid is a generalization of a matroid. Note that if $(E,\mu)$ is a valuated matroid, then $(E, a\mu)$ is also a valuated matroid, for any $a\in \R_{>0}$. Furthermore, for any real number $b$, it follows that $(E,\mu+b)$ is a valuated matroid as well. 

\begin{Exa}[Uniform matroid]
    Let $E=[n]\coloneqq\{1,2,\dots,n\}$, and let $d \leq n$ be a natural number. Define $\mu: \binom{E}{d}\to \T$ as a constant function, for example, say $\mu(B)=1_\T$ for each $B\in\binom{E}{d}$. It is immediate to see that $\mu$ satisfies axiom \ref{axiom: valuated_matroid}. Thus, the pair $(E,\mu)$ is a valuated matroid. For a more specific and geometric description, see \cite[Example. 2.1]{Fink_Olarte_2022}.
\end{Exa}

A linear space $V\subset K^n$ of dimension $d$ is determined by its Pl\"{u}cker coordinate vector $P \in K^{\binom{[n]}{d}}$.  Applying the valuation to each component of $P$ yields a vector $\Vtrop(P) \in \T^{\binom{[n]}{d}}$ that can be shown to satisfy Axiom \ref{axiom: valuated_matroid}.  Valuated matroids arising in this way are said to be \emph{realisable} over the valued field $K$.

The \emph{Dressian} $\mathit{Dr}(n,d) \subset \P\T^{\binom{[n]}{[d]}}$ is the space of all valuated matroids up to tropical scaling.  For any valued field $K$, there is the subspace of all valuated matroids that are realizable over $K$, and this is called the \emph{tropical Grassmannian} $\mathit{Gr}^{\trop}(n,d)$.  See, for instance, \cite{Speyer_Sturmfels_2004}, \cite{Herrmann_Joswig_Speyer_2014}, and \cite{Iezzi_Schleis_2023}.

A valuated matroid $p \in \T^{\binom{[n]}{d}}$ determines a submodule $L(p)\subset \T^n$, called its associated \emph{tropical linear space}, as follows. For each subset $A \in \binom{[n]}{d-1}$ we have the vector $\alpha_A \in \T^n$ given by
\[
(\alpha_A)_i = p_{A\cup \{i\}};
\]
this is called the \emph{valuated co-circuit vector} associated with $A$.  The tropical linear space $L(p)$ is the $\T$-submodule of $\T^n$ spanned by the set of all valuated co-circuit vectors. It can be shown out that $L(p)$ determines $p$ up to tropical multiplication by a scalar.  Hence we can talk about realizable and non-realizable tropical linear spaces just as we do for valuated matroids. One can also construct $L(p)$ as the intersection of the tropical hypersurfaces associated with the \emph{valuated circuit vectors}
\[
\beta_B = \bigoplus_{i\in B} p_{B \smallsetminus \{i\}} \odot X_i \in \T[X_1, \ldots, X_n]
\]
for each $B\in \T^{\binom{[n]}{d+1}}$. 

A foundational fact about linear spaces and their tropicalizations is that, if we start with a linear space $V \subset K^n$, then the tropicalization of $V$ as a variety coincides with the tropical linear space $L\big(\trop(P)\big)$ determined by the tropicalization of $P$; see 
\cite[Prop.\ 4.4.4]{maclagan2015introduction}.  This can be thought of as a linear version of the Fundamental Theorem.
 For further details on tropical linear spaces and valuated matroids, we refer the reader to \cite{Speyer_2008}, \cite[\S~5.5]{maclagan2015introduction}, and \cite{Frenk}.

\subsection{Tropical ideals}

We are now ready to return to the task of examining and generalizing the ideals of $\T[X_1,\dots,X_n]$ of the form $\trop(J)$.

Given an ideal $I \subseteq \T[X_1,\dots,X_n]$, and $d\ge 0$ we denote by $I_{\le d}$ the subset of $I$ consisting of polynomials with degree at most $d$.
     
\begin{definition}\label{def: tropical_ideal}
    A \emph{tropical ideal} $I\subseteq \T[X_1,\dots,X_n]$ is an ideal such that $I_{\le d}$ is a tropical linear space for each $d\ge 0$.
\end{definition}
 Equivalently, a tropical ideal $I$ is an ideal that satisfies the \emph{vector elimination axiom} of a valuated matroid: For any $f=\bigoplus_{\alpha}f_\alpha\odot X^\alpha$ and $g=\bigoplus_\alpha g_\alpha\odot X^\alpha$ in $I_{\le d}$ such that  $f_\beta = g_\beta\neq \infty$, there exist a polynomial
\[
    h =\bigoplus_\alpha h_\alpha\odot X^\alpha\in I_{\le d} \quad \text{ with } \quad
    \begin{cases}
        h_\beta = \infty \\
        h_\alpha \ge f_\alpha\oplus g_\alpha \textup{ for all } \alpha,\textup{ with equality whenever } f_\alpha\neq g_\alpha
    \end{cases}
\]
(see \cite[Def. 1.1]{maclagan2018tropical}).  

Given an ideal $J \subseteq K[X_1,\dots,X_n]$, its tropicalization $\Vtrop(J)$ is a tropical ideal; see \cite[pg. 641]{maclagan2018tropical}.  On the other hand, there exist \emph{non-realizable} tropical ideals, i.e., tropical ideals that are not of the form $\Vtrop(J)$; see \cite[Example 2.8]{maclagan2018tropical} and \cite[Theorem A]{FinkGiansiracusa2X}. 

Tropical ideals have many interesting properties that make them either very nice to work with or frustrating, depending on your angle. Here we present a few of them.

General ideals in $\T[X_1, \ldots, X_n]$ might not be finitely generated, and the class of all ideals does not satisfy the ascending chain condition.  
For tropical ideals, the question of finite generation is even worse: they are almost never finitely generated as ideals. The tropicalization of the ideal $\langle X - 1 \rangle$ demonstrates this, where for each $d$ we see that $X^d \oplus 0$ is present but is not in the ideal generated by elements in strictly lower degrees.  Worse still, a tropical ideal need not be determined by any of its truncations to a finite degree, as shown in \cite[Example 3.10]{maclagan2018tropical}.   This poses an important set of question: 

\begin{question}
How can we construct non-realizable tropical ideals using a finite amount of information?
\end{question}
\begin{question}
How can we work with tropical ideals (beyond the realizable case) in computer algebra systems?
\end{question}

On the other hand, tropical ideals 
satisfy the ascending chain condition; i.e., there is no infinite ascending chain of tropical ideals 
\[
    J_1 \subsetneq J_2 \subsetneq J_3 \subsetneq \cdots 
\]
in $\T[X_1, \dots, X_n]$; see \cite[Thm. 3.11]{maclagan2018tropical}, where this was proven by developing a Gr\"{o}bner theory for tropical ideals.

Given a tropical ideal $I$, it has a Hilbert function defined using ranks of valuated matroids (or dimensions of tropical linear spaces) in place of dimensions of linear spaces), and it follows by an argument using the ACC that the Hilbert function eventually agrees with an integer-coefficient polynomial function \cite[Thm. 3.8]{maclagan2018tropical}. Moreover, tropicalization of homogeneous ideals preserves the Hilbert function; i.e., the Hilbert functions of an ideal $J \subset K[X_1,\dots,X_n]$ and the tropical ideal $\trop(J)$ are identical.

For a tropical ideal $I$, the associated tropical variety $V^\Vtrop(I)$ is the intersection of the tropical hypersurfaces $V^\Vtrop(f)$ for $f\in I$. A striking theorem (\cite[Thm. 6.6]{maclagan2022varieties}) says that $V^\Vtrop(I)$ is a balanced weighted polyhedral complex.  

One might be tempted to declare that the right definition of `tropical variety' is anything coming from a tropical ideal.  However, the results of \cite{draisma2021tropical} should at least give one pause.  They show that there exist tropical linear spaces $L \subset \T^n$ that are not the variety of any tropical ideal. Indeed, let $V_8$ and $\mathcal{U}_{2,3}$ be the V\'{a}mos matroid and a uniform matroid respectively. Then the direct sum $V_8\oplus\mathcal{U}_{2,3}$ gives us such an example.

The V\'{a}mos matroid is non-realizable, and hence so is the sum $V_8\oplus\mathcal{U}_{2,3}$.  Hence one can ask the following.

\begin{question}
Is there a tropical ideal $I$ such that $V^\Vtrop(I)$ is a non-realizable tropical linear space?
\end{question}

A tropical ideal $I$ and its associated bend congruence $\Bb(I)$ are equivalent packagings of the same information, and the set of homomorphisms
$\T[X_1, \ldots, X_n]/\Bb(I) \to \T$
is precisely the tropical variety $V^\Vtrop(I)$.  Hence we should consider
the algebra 
\[
\T[X_1, \ldots, X_n]/\Bb(I)
\]
as the algebra of functions of the tropical \emph{scheme}.  With an appropriate notion of $\Spec$ for $\T$-algebras, we can consider these to be the local affine models for constructing general non-affine \emph{tropical schemes}.  A fascinating challenge that we will not discuss here is:

 \begin{question}
 What is the appropriate notion of gluing these affine pieces together?
 \end{question}

\section{Second lecture}\label{sec:SecondLecture}

In this lecture we will explore the relationship between the scheme-theoretic perspective on tropicalization via bend relations and Berkovich's analytic spaces.

\subsection{The story for linear spaces}
Let $K$ be a field, $v \co K\to \T$ valuation, and $|a| \coloneqq e^{-v(a)}$ the corresponding norm. For an $n$-dimensional $K$-vector space $V$, a \emph{$K$-semivaluation} on $V$ is a map $w\co V\to \T$ such that $w(a+b) \geq \min\{w(a),w(b)\}$ and $w(\lambda a )= v(\lambda) \odot w(a)$.

\begin{lemma}
Let $a,b \in V$ with $w(a) \neq w(b)$. Then $w(a+b) = \min\{w(a),w(b)\}$.
\end{lemma}
In particular, we observe that for any triple $a,b,c$ satisfying $a+b+c=0$ we have that the minimum of $w(a),w(b)$, and $w(c)$ is attained at least twice. Thus, $K$-semivaluations are solutions to certain bend relations. In \cite[\S~7]{bkkuv_buildings} it was shown, that these are the points of a certain tropicalization. In the following, we will sketch this idea.

Given a vector space $V$, let us consider the rather larger vector space $\widehat{V}$ with basis given by the underlying set of $V$.  I.e., $\widehat{V} \coloneqq \bigoplus_{a \in V} K$, and we write the basis as $\{x_a\mid a\in V\}$. There is a canonical projection map 
\[
\pi\co \widehat{V}\to V
\]
defined by sending $x_a \mapsto a$, and hence there is a linear isomorphism $V\cong \widehat{V}/\ker(\pi)$. 
We will be tropicalizing this presentation of $V$.  This might appear counter to the way we usually tropicalize subspaces of a fixed space (a torus of affine space).  However, we will take the scheme-theoretic dual perspective here.  The large vector space $\widehat{V}$ is the space of linear functions on an ambient space (analogous to the polynomial algebra $K[X_1, \ldots, X_n]$), and $V$ is the space of linear functions on a sub-object $Z$ (analogous to $K[X_1, \ldots, X_n]/J$); the map $\pi$ identifies functions that agree on $Z$.  The tropicalization of this geometric picture is given by the $\T$-module
\[
\Vtrop_\pi(V) \coloneqq \left( \bigoplus_{a\in V} \T \right) / \Bb\big(\trop(\ker(\pi))\big).
\]
Understanding the structure of the tropicalization clearly requires a description of the kernel of $\pi$.

\begin{proposition}\label{prop:ker-pi-spanning-set}
The kernel of $\pi$ is spanned by the elements:
\begin{itemize}
\item $x_a+x_b+x_c$ for $a+b+c=0$, and
\item $x_{\lambda a} - v(\lambda)x_a$ for $\lambda\in K, a\in V$. 
\end{itemize}
\end{proposition}

The kernel of $\pi$ is a linear subspace of $\widehat{V}$ (albeit infinite dimensional if $K$ is an infinite field), and so the tropicalization of $\ker(\pi)$ is a tropical linear space in $\widehat{V}$ (corresponding to a valuated matroid on ground set $V$).  In general, tropicalizing a spanning set for a linear space $L$ does not yield a generating set for the tropical linear space $\Vtrop(L)$, and the bend relations of a spanning set need not generate the bend congruence of $\Vtrop(L)$.  However, in the case of the spanning set Proposition \ref{prop:ker-pi-spanning-set}, this is true.

\begin{proposition}\label{prop:bend-rels-generate}
The congruence $\bend \big(\Vtrop (\ker(\pi))\big)$ is generated by:
\begin{itemize}
\item $\bend(x_a\oplus x_b \oplus x_c )$ for $a+b+c=0$, and
\item $x_{\lambda a} \sim v(\lambda)x_a$ for $\lambda\in K, a\in V$. 
\end{itemize}
\end{proposition}
\begin{proof}
The argument is by elementary algebraic manipulations analogous to \cite[Prop. 3.4.3]{Giansiracusa2x2022}
\end{proof}

\begin{corollary}\label{Cor: semi_val_trop_linear_space}
    The points of $\Vtrop(V)$, i.e., morphisms $\Vtrop(V) \to \T$, are exactly the $K$-semivaluations on $V$. Hence, the space of semivaluations on $V$ is a tropical linear space.
\end{corollary}
\begin{proof}
A $\T$-valued point $w$ sends each vector $x_a$ to a tropical number $w(a)$.  The relations $x_{\lambda a} \sim v(\lambda)\odot x_a$ encodes the condition that $w$ is compatible with the valuation $v$ on $K$, and the relations $\bend(x_a\oplus x_b \oplus x_c )$ precisely encode the condition that $w$ satisfies the ultra-metric triangle inequality.
\end{proof}

We note that the valuated matroid whose associated tropical linear space is the space of semivaluations has ground set $V$, which is infinite in general. In \cite{bkkuv_buildings} it was shown, that this tropical linear space is the limit of all finite tropicalizations of $V$. This can be easily re-proven from our perspective of bend relations, as we now explain.

Consider a finite set $E$ and a surjective linear map $q: K^E \twoheadrightarrow V$.  We are thinking of these as spaces of linear functions of the dual geometric objects.  Taking linear duals, this corresponds to a variety (linear subspace) $Z=V^* \subset (K^E)^* \cong K^E$.  The tropicalization of this picture is the $\T$-module
\[
\Vtrop_q(V) \coloneqq \T^E / \Bb\big(\trop(\ker(q))\big),
\]
and the $\T$-linear dual of this gives us the tropical linear space 
$\Vtrop(Z) \subset \T^E$.  These tropicalizations are functorial in the sense that a diagram
\begin{center}
\begin{tikzcd}
    K^{E_1} \ar[rr] \ar[dr,"q_1" '] & & K^{E_2} \ar[dl, "q_2"] \\
    & V
\end{tikzcd}
\end{center}
in which the top arrow is induced by a map of sets $E_1 \to E_2$ induces a $\T$-linear map of tropicalizations $\Vtrop_{q_1}(V) \to \Vtrop_{q_2}(V)$.

\begin{theorem}
The colimit of all tropicalizations of finite presentations $K^E \to V$ is canonically isomorphic to the tropicalization of $\pi\co \widehat{V} \to V$.
\end{theorem}
\begin{proof}
Given a presentation $q\co K^E \to V$, there is an induced map $\widehat{q}\co K^E \to \widehat{V}$ sending $x_e \mapsto x_{q(e)}$ for each $e\in E$.  These maps are all compatible as $q$ varies, and so they induce a canonical map from the colimit to the big tropicalization $\Vtrop_\pi(V)$.  We will show it is an isomorphism.

It is clearly surjective since for any $a\in V$ and finite presentation $K^E \to V$, we have the diagram
\begin{center}
\begin{tikzcd}
    K^{E} \ar[rr] \ar[dr,"q" '] & & K^{E \cup \{a\}} \ar[dl, "q'"] \\
    & V
\end{tikzcd}
\end{center}
where $q'$ is $q$ extended by sending the basis vector $x_a$ to $a\in V$, and so the generator $x_a \in \bigoplus_{a\in V} \T$ is in the image.

The argument for injectivity is similar. It suffices to show that the map from the colimit of the system of congruences  $\Bb\big(\Vtrop(\ker(q))\big) \subset  \T^E \times \T^E$ (for $E$ finite) to $\Bb\big(\Vtrop(\ker(\pi))\big)$ is surjective.  Moreover, since the latter is generated by the relations given in Proposition \ref{prop:bend-rels-generate}, it suffices to check that $x_a\oplus x_b \oplus x_c$ is in the image of some $\Vtrop(\ker(q))$ whenever $a+b+c = 0$ in $V$, and likewise for $x_{\lambda a} \oplus v(\lambda)\odot x_a$.   Given a triple $a,b,c \in V$ with $a+b+c=0$, we can extend $\{a,b,c\}$ to a spanning set $E$ and consider the resulting presentation $q\co E \to V$.  We have $x_a + x_b + x_c \in \ker(q)$, and so the corresponding tropical sum $x_a \oplus x_b \oplus x_c \in \Bb\big(\Vtrop(\ker(\pi))\big)$ is in the image of $\Bb\big(\Vtrop(\ker(q))\big)$.  For generators of the second type, we use the same idea. Given $a\in V$ and $\lambda \in K$, choosing a spanning set $E \subset V$ containing $a$ and consider the resulting presentation $q\co K^E \to V$.  Then $x_{\lambda a} - \lambda x_a \in \ker(q)$ and so $x_{\lambda a} \oplus v(\lambda)\odot x_a \in \trop(\ker(q))$. 
\end{proof}

\subsection{The non-linear story: Berkovich analytification}

We now sketch some results of \cite{Giansiracusa2x2022}.

Let $K$ be a valued field with valuation $v \co K\to \T$, and $X=\Spec A$ an affine $K$-scheme, where $A$ is a $K$-algebra.
The usual definition of a \emph{multiplicative semivaluation} on $A$ is equivalent to the following: it is a map $w \co A\to \T$ such that $w(a)\oplus w(b) \oplus w(c)$ bends if $a+b+c=0 \in A$, and $w(ab)=w(a)\odot w(b)$.
We say that a multiplicative semivaluation is \emph{compatible} with the valuation on $K$ if $w(\lambda a)=v(\lambda) \odot w(a)$ for $\lambda\in K$ and $a\in A$.  Multiplicative semivaluations compatible with $v$ will simply be called $v$-semivaluations.

The advantage of defining semivaluations in the way above is that we immediately see how to extend from totally ordered idempotent semifields like $\T$ to arbitrary idempotent semirings.  This turns out to be a conceptually useful thing to do

\begin{definition}\label{def: Ber_ana} 
Given an affine $K$-scheme $X = \Spec A$, the underlying set of the \emph{Berkovich analytification} $X^{\an}$ of $X$ is the set of $v$-semivaluations on $A$.
\end{definition}

When $X$ is not affine, one can choose an open affine covering, construct the analytification of each affine patch, and then glue these together.

In Berkovich's theory \cite{Berkovich_book}, the underlying sets are just the starting point for constructing his category. They are equipped with a topology and a sheaf of analytic functions.  The topology can be described in terms of our tropical algebraic perspective, but it seems that the structure sheaf genuinely lives outside of the realm of tropical geometry.

\begin{Exa}
Due to their fuzzy nature, Berkovich spaces are difficult to visualize, especially in higher dimensions. In \autoref{fig:berkovich_line} we see a visualization of the Berkovich projective line $(\P^{1})^{\an}$ from \cite{baker_berkovich}.  

\begin{figure}
    \centering
    \includegraphics[width=0.5\linewidth]{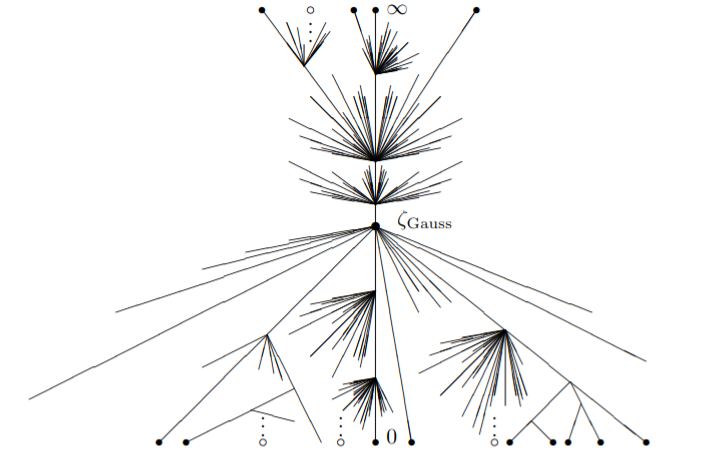}
    \caption{The Berkovich projective line $(\P^{1})^{\an}$, adapted from an illustration of Joe Silverman.}
    \label{fig:berkovich_line}
\end{figure}
\end{Exa}

The following theorem of Payne was an important early result in the development of tropical geometry.

\begin{theorem}[\cite{payne_analytification}]\label{thm:payne_analytification}
Let $X$ be a quasiprojective $K$-scheme.
\begin{enumerate}[label=\arabic*)]
    \item For any embedding of $X$ into a quasiprojective toric variety, there is a canonical surjection of $X^{\an}$ onto the corresponding tropicalization;
    \item\label{item2payne} The analytification is homeomorphic to the limit of all tropicalizations, i.e.\
\[
    X^{\an} \cong \lim_{\substack{\varphi\co  X\hookrightarrow Y \\ Y \text{ toric}}} \Vtrop_{\varphi}(X)\,.
\]
\end{enumerate}
\end{theorem}

We are interested in proving a scheme-theoretic version of this theorem. To keep the presentation simple, we will restrict our attention to the affine case. Our notation will differentiate between the set-theoretic tropicalization  $\Vtrop_\varphi(X)\subseteq \T^n$ and the scheme-theoretic tropicalization $\Strop_\varphi(X)$, that we introduce now: 

\begin{definition}[\cite{Giansiracusa2X_2016}]
    Let $X= \Spec A$ be an affine $K$-scheme. Let $M$ be an integral monoid with zero (i.e.\ an $\mathbb F_1$-algebra) and let $\varphi: K[M] \twoheadrightarrow A$ be a surjective $K$-algebra morphism. Then $\Strop_\varphi(X)$ is defined as follows: 
    \[
        \Strop_\varphi(X)\coloneqq \Spec \big( \T[M]/\bend \trop (\ker(\varphi)) \big)
    \]
\end{definition}

The scheme-theoretic enrichment of Payne's inverse limit theorem can be stated as follows.
\begin{theorem}\label{thm:limits}
Let $X=\Spec A$ be an affine $K$-scheme.
\begin{enumerate}

\item[1)] The category of closed embeddings $X\hookrightarrow Y=\Spec K[M]$, for $M$ a monoid of monomials with respect to which we can tropicalize, has an initial object $u\co X\hookrightarrow \widehat{X}$.

\item[2)] There is a canonical bijection $\Vtrop_u(X)=X^{\an}$.

\item[$2'$)] The set of $\T$-valued points of the $\T$-scheme $\Strop_u(X)$ is $X^{\an}$.

\item[$2''$)] For any $\T$-algebra $S$ there is a canonical bijection
\[\Strop_u(X)(S) \cong \{\text{$v$-semivaluations on $A$ with values in $S$} \}.\]

\item[$2'''$)]   There is a universal $v$-semivaluation on $A$; it takes
values in the $\T$-algebra $\mathcal{V}$ of regular functions on 
$\Strop_u(X)$, and it is universal in the sense that any $v$-semivaluation on
$A$ taking values in $S$ factors uniquely through a $\T$-algebra homomorphism
$\mathcal{V} \to S$.

\item[3)] There is a natural isomorphism of $\T$-schemes: 
\[
    \Strop_u(X) \cong \lim_{ X\xhookrightarrow{\varphi} \mathbb{A}^n} \Strop_{\varphi} (X)\, .
\]
\end{enumerate}
\end{theorem}

The initial embedding $X \to \widehat{X}$ from \autoref{thm:limits} part (1) is given at the level of algebras by
\[
    \begin{tikzcd}[row sep=0pt,/tikz/column 1/.append style={anchor=base east},/tikz/column 2/.append style={anchor=base west}]
        A & \ar[l, twoheadrightarrow, "\ev" ']  K[A]\\
        a & \ar[l, mapsto]x_{a},
    \end{tikzcd}
\]
where we consider $A$ as a commutative monoid with respect to multiplication.
The map $K[A] \to A$ is called an \emph{evaluation} map because it takes a formal $K$-linear combination of elements of $A$ and evaluates it to an element of $A$ using the arithmetic operations of $A$ as a $K$-algebra.
The universal property of this being an initial object means that, for any commutative monoid $M$ and morphism $f\co K[M]\to A$, there is a unique morphism
\begin{center}
    \begin{tikzcd}
        K[M] \ar[r, "f"] \ar[d, "\exists !" ', dashed] & A \\
        K[A] \ar[ur] & 
    \end{tikzcd}
\end{center}
given by $m \mapsto x_{f(m)}$.
One can prove that the kernel of $\ev$ is generated by the following elements: 
\begin{align*}
    x_a+x_b+x_c & \quad \quad \text{for $a+b+c=0$ in $A$} \\
    x_{\lambda a} - \lambda x_a & \quad \quad  \text{for $a \in A$, $\lambda \in K$}.
\end{align*}

With more effort from the fact above it follows that the congruence $\Bb\big(\trop (\ker(\ev))\big)$ is generated by the bend relations of the tropicalization of the elements generating $\ker(\ev)$, i.e.\ by: 
\begin{align*}
    x_a \oplus x_b \oplus x_c \sim x_a \oplus  x_b \sim x_b \oplus x_c \sim x_a \oplus x_c & \quad \quad \text{for $a+b+c=0$ in $A$} \\
    x_{\lambda a} \sim \lambda x_a & \quad \quad  \text{for $a \in A$, $\lambda \in K$}.
\end{align*}
We can see that the relations above encode the axioms of a generalized $v$-semivaluation on $A$: the former relations encode the (generalized) ultrametric triangle inequality, and the latter encode the compatibility with the valuation on $K$. The algebra $\mathcal{V}$ of global functions on $\Strop_u(X)$ is given by
\[
    \mathcal V = \T[A]/\bend\big(\hspace{-2pt}\trop (\ker(\ev))\big)
\] 
and from here it is straighforward to see that the $v$-semivaluation
\[
    \begin{tikzcd}[column sep = 17pt]
        A \ar[r] & \mathcal V,
    \end{tikzcd}
\]
defined by sending an element $a$ to $x_a$, is universal among all the $v$-semivaluations on $A$.

\section{Third lecture}
    \label{sec:ThirdLecture}

In this lecture we will begin to explore what happens when we try to tropicalize objects defined by quotients of algebras of non-commutative polynomials.  This is mostly a report on work that was done by Beth Zhou and Luca Chadwick when they were undergraduates at Durham.

\subsection{Non-commutative tropicalization}

Recall that given a cancellative monoid $M$, the process of tropicalization associates to the monoid algebra $K[M]$ the tropical monoid algebra $\T[M]$. Given an ideal $I \subseteq K[M]$, valuating the coefficients yields a tropical ideal $\trop(I) \subset \T[M]$ (this requires that $M$ is cancellative). The quotient algebra $K[M]/I$ is associated on the tropical side with the quotient $\T[M]/\bend\big(\trop(I)\big)$.

Here is an important observation:  the monoid $M$ does not need to be a commutative!  In the extreme non-commutative case, we can consider the free associative monoid of words on $n$ letters $x_1, \dots , x_n$, where the product is given by concatenation of words and the unit is the empty word.  This free monoid is denoted $F_n$. The associated monoid algebra $K[F_n]$ is traditionally written as $K\langle x_1, \dots , x_n \rangle$; this is algebra of noncommutative polynomials over $K$.
Notice that 
\[
    K\langle x_1, \dots , x_n \rangle = \bigoplus_{d=0}^\infty (K^n)^{\otimes d};
\]
i.e., the noncommutative polynomial algebra is the same as the tensor algebra $T(K^n)$.

Given a $K$-algebra $A$, a choice of generating set is equivalent to a surjective homomorphism $\pi \co K \langle x_1, \dots , x_n \rangle \twoheadrightarrow A$.  The kernel of $\pi$ is a two-sided ideal of $K \langle x_1, \dots , x_n \rangle$, and its tropicalization $\trop(\ker(\pi))$ is a two-sided ideal of the tropical noncommutative polynomial algebra $\T \langle x_1, \dots , x_n \rangle $ and a tropical linear space --- i.e., it is a 2-sided tropical ideal. Thus, in this non-commutative setting we can define the tropicalization of $A$ with respect to $\pi$ as $\T \langle x_1, \dots , x_n \rangle  / \bend\big(\trop(\ker(\pi))\big)$, in complete analogy with what we did in the commutative setting. 

\begin{proposition}\label{prop:non-comm-trop}
    Fix a valued field $v \co K \rightarrow \T$. Then:
    \begin{enumerate}[label=(\arabic*)]
        \item The tropicalization of $K[x_1, \dots , x_n]$ with respect to the map 
        \[
            \begin{tikzcd}[column sep = 17pt]
                 K \langle x_1, \dots , x_n \rangle \ar[r, twoheadrightarrow] & K[x_1, \dots , x_n]
            \end{tikzcd}
        \] 
        is the usual tropical commutative polynomial algebra $\T[x_1, \dots , x_n]$; 
        \item The tropicalization of the algebra of $n\times n$ matrices,  
        $\textup{Mat}_{n \times n}(K)$, with respect to the presentation 
    \[
        \begin{tikzcd}[column sep = 17pt]
            K \langle x_{i,j} \mid i,j \in \{1, \dots , n\} \rangle \ar[r, twoheadrightarrow] & \textup{Mat}_{n \times n}(K)
        \end{tikzcd}
    \]
        given by
    \[
            x_{i,j} \mapsto 
                    \begin{matrix}
                    & j & \\
                        & \begin{pmatrix}
                            0 & \dots & 0\\
                         \vdots & 1 & \vdots \\
                            0 & \dots & 0
                        \end{pmatrix} & i 
                    \end{matrix}
    \]
    is the algebra of tropical matrices, $\textup{Mat}_{n \times n}(\T)$;
    
    \item\label{axiom3: prop:non-comm-trop}  The tropicalization of the exterior algebra $\bigwedge K^n$ with respect to the quotient map 
    \[
        \begin{tikzcd}[column sep = 17pt]
            K \langle x_1, \dots , x_n \rangle \ar[r, twoheadrightarrow] & \bigwedge K^n
        \end{tikzcd}
    \]
    is the tropical algebra $\bigwedge_{\trop} \T^n \coloneqq \T[x_1, \dots , x_n] / (x_i^2 \sim \infty)$ as proposed in \cite{Giansiracusa2x2018}.  Moreover, given a surjective linear map $\pi: K^n \twoheadrightarrow Q$, the tropicalization of the composition
    \[
        \begin{tikzcd}[column sep = 17pt]
            K \langle x_1, \dots , x_n \rangle \ar[r, twoheadrightarrow] & \bigwedge K^n \ar[r, twoheadrightarrow] & \bigwedge Q
        \end{tikzcd}
    \]
    is the push-out of the diagram
    \[
        \begin{tikzcd}
            \T^n \ar[r] \ar[d] & \bigwedge_{\trop} \T^n \\
            \T^n / \Bb\big(\Vtrop(\ker (\pi))\big).  & 
        \end{tikzcd}
    \]

    \end{enumerate}
    \end{proposition}

\begin{proof}
We leave this as an exercise for the motivated reader.  The arguments are entirely elementary.
\end{proof}

\begin{remark}
    The degree $d$ part of $\bigwedge_{\trop} \T^n$ is a free $\T$-module of rank $\binom{n}{d}$. Notice that the square of an element $a \in \bigwedge_{\trop} \T^n$ need not to be equal to $\infty$. E.g., in two variables:
    \[
        (x_1\oplus x_2)^2 = x_2^2 \oplus  x_1^2 \oplus  x_1\odot x_2 \oplus  x_2\odot x_1 = x_1\odot x_2.
    \]
\end{remark}

\subsection{The classical Pl\"{u}cker embedding story} 
One of the many very useful things that exterior algebras do is provide a framework for describing the Pl\"{u}cker embedding of the Grassmannian 
$\mathit{Gr}(n,d)$ into projective space 
$\mathbb{P}\left( \bigwedge^d K^n \right)$.  We first recall how this works, before moving on to the tropical analogue of this story in the next section.

Given a $d$-dimensional subspace $Q \subseteq K^n$, we can decompose $K^n$ as $Q \oplus Q^\perp$ and by projecting we obtain a surjection $K^n \rightarrow Q$, thus a surjection 
\[
    \begin{tikzcd}[column sep = 17pt]
        {\bigwedge}^d K^n \ar[r, twoheadrightarrow] & {\bigwedge}^d Q.
    \end{tikzcd}
\]
As $Q$ has dimension $d$, the target is 1-dimensional and we obtain a point $p_Q \in \mathbb P (\bigwedge^d K^n)$. Here $\mathbb P (\bigwedge^d K^n)$ denotes to the space of $1$-dimensional quotients of $\bigwedge^d K^n$ (rather than 1-dimensional subspaces).

One can recover $Q$ from the point $p_Q$ as follows. Consider the map
\[
    \begin{tikzcd}[column sep = 17pt]
        - \wedge p_Q \co K^n = {\bigwedge}^1 K^n \ar[r] & {\bigwedge}^{d+1}K^n;
    \end{tikzcd}
\]
it can be shown that  $\ker(- \wedge p_Q) = Q^\perp$, and so $\ker(- \wedge p_Q)^\perp = Q$.

In general, given a vector $p \in \mathbb P (\bigwedge^d K^n)$, the following conditions are equivalent: 
\begin{enumerate}
    \item  There exists a subspace $Q$ such that $p = p_Q$;
    \item  $p$ satisfies the Plücker relations;
    \item $\textup{dim} \big(\ker(- \wedge p)\big) = n-d$;
    \item  $\textup{rank}(- \wedge p) = d$. 
\end{enumerate}

\subsection{The tropical Pl\"{u}cker embedding story}
Since $\bigwedge^d_{\trop} \T^n$ is a free module with basis given by the size $d$ subsets of $\{1, \ldots, n\}$, one can think of valuated matroids (in their tropical Pl\"{u}cker vector formulation) of rank $d$ as living here. 

Given a valuated matroid $p$, the corresponding  tropical linear space $L_p \subseteq \T^n$ is
\begin{align*}
    L_p  :&= \textup{span}(\{\text{valuated cocircuit vectors of $p$}\})  \\
    & = \{\text{valuated circuit vectors of $p$}\}^\perp
\end{align*}
i.e., the valuated circuits of $p$ are the linear equations defining $L_p$; each valuated circuit in fact is a linear form $\T^n \rightarrow \T$, and points of $L_p$ are those points $x \in \T^n$ at which each of the linear forms above tropically vanishes.  In the classical story, subspaces and quotients are interchageable thanks to orthogonal projection.  In the tropical world, subspaces and quotients have somewhat different characters.  

In other words, using a notation recalling that of the classical case, define
\[
Q_p \coloneqq \T^n/ \bend(\{\text{circuits of $p$}\}) = \T^n/ \bend(L_p^\perp),
\]
where $L_p^\perp$ is the tropical linear space dual to $L_p$. Note that $L_p= Q_p^\vee$. 
 
\begin{conjecture}
    $L_p^\vee = Q_p$.
\end{conjecture}

    Notice that the definition of $Q_p$ makes sense even for vectors $p \in \bigwedge^d_{\trop} \T^n$ that do not satisfy the tropical Plücker relations since we can still define the circuit vectors via the recipe
    \[
    \bigoplus_i p_{A \smallsetminus \{i\}} \odot x_i
    \]
    for $A \subset [n]$ of size $d+1$.

Proposition \ref{prop:non-comm-trop} part \ref{axiom3: prop:non-comm-trop} tells us how tropicalize the exteior algebra of a quotient $K^n \twoheadrightarrow Q$ by giving a recipe in terms of the tropicalization of $Q$.  This strongly suggests a natural generalization of tropical exterior algebra to non-realizable quotient modules.

\begin{definition}
    Given a surjective $\T$-module homomorphism $\T^n \twoheadrightarrow Q$, we define the tropical exterior algebra $\bigwedge_{\trop} Q$ as the push-out of the following diagram of $\T$-modules: 
    \begin{center}
        \begin{tikzcd}
            \T^n \ar[r] \ar[d] & Q. \\
            \bigwedge_{\trop} \T^n & 
        \end{tikzcd}
    \end{center}
\end{definition}

We are now ready to state the main theorem of \cite{Giansiracusa2x2018}, which gives a tropical analogue of the classical Pl\"{u}cker embedding story in terms of tropical exterior algebra.

\begin{theorem}\label{theorem:tropical-plucker}
Let $p \in \bigwedge^d_{\trop} \T^n$. 
    \begin{enumerate}[label=(\arabic*)]
        
        \item  The vector $p$ satisfies the tropical Plücker relations if and only if $\bigwedge_{\trop}^d Q_p$ is a free $\T$-module of rank $1$.
        
        \item \label{point:circuits} If $p$ satisfies the tropical Plücker relations, then the components of the map 
        \[
            \begin{tikzcd}[column sep = 17pt]
                - \wedge p \co  \bigwedge_{\trop}^1 \T^n \ar[r] & \bigwedge_{\trop}^{d+1} \T^n
            \end{tikzcd}
        \]
        are linear forms on $\T^n$ that are equal to the circuits of $p$.
        
        \item If $p$ satisfies the tropical Plücker relations, then the associated tropical linear space $L_p$ is equal to the set
        \[
            \{x \in \T^n \mid \text{each component of $- \wedge p$ tropically vanishes at $x$}\}.
        \]
        
        \item The isomorphism $\bigwedge_{\trop}^d\T^n \cong \bigwedge_{\trop}^{n-d}\T^n$ implements duality of (valuated) matroids.
        
    \end{enumerate}
\end{theorem}
\begin{remark}
    Notice that in general one can define the circuits of any vector $p \in \bigwedge^d_{\trop} \T^n$ as the components of the map as in part \ref{point:circuits} of \autoref{theorem:tropical-plucker} above, even when $p$ does not satisfy the tropical Plücker relations.
\end{remark}

\subsection{Clifford algebras}
Let $V$ be a vector space over $K$, and $q \co V \rightarrow K$ a quadratic form. 
\begin{definition}
    The Clifford algebra $Cl(V,q)$ is the $\Z/2\Z$-graded algebra 
    \[
        Cl(V,q) \coloneqq T(V)/\big(u \otimes u = q(u)\big),
    \]
    where $T(V)$ denotes the tensor algebra on $V$.
    Over $\C$ all non-degenerate quadratic forms are equivalent, so for $V = \C^n$ and $q = \textup{I}_n \in \textup{Mat}_{n \times n }(\C)$, we denote $Cl(\C^n, q)$ as $C_n$.
\end{definition}

By explicit computations one can easily show the following results.

\begin{proposition} The following isomorphisms hold:
    \begin{enumerate}[label=(\arabic*)]
        \item $C_0 \cong \C$;
        \item $C_2 \cong \textup{Mat}_{n \times n}(\C)$;
        \item $C_m \otimes C_n \cong C_{n+m}$; 
        \item \label{point:periodicity} $C_n$ is Morita equivalent to $C_{n+2}$ for every $n \in \N$.
    \end{enumerate}
\end{proposition}
\begin{remark}
    Bott periodicity is a deep and foundational theorem in homotopy theory. It says that the space $\mathbb{Z}\times BU$ representing complex $K$-theory is a 2-\emph{periodic} infinite loop space.  This means that if we apply the based loop space functor to it twice then we get the original space back.  One of the classic proofs of Bott Periodicity \cite{Bott-periodicity} begins with the periodicity of complex Clifford algebras.
\end{remark}
\begin{definition}
    Let $C_n^{\trop}$ be the tropicalization of the presentation $T(\C^n) \twoheadrightarrow C_n$, for every $n \in \N$. 
\end{definition}
\begin{proposition}
For tropicalized Clifford algebras, we have:
    \begin{enumerate}
        \item $C_m^{\trop} \otimes C_n^{\trop} \cong C_{m+n}^{\trop}$.
        \item $C_2^{\trop}$ is not isomorphic to $\textup{Mat}_{2 \times 2}(\T)$. The following isomorphism instead holds: 
        \[
            C_2^{\trop} \cong \T[x,y]/\{x^2 =  0, y^2 = 0\}
        \]
        i.e., $C_2^{\trop}$ is isomorphic to the tropicalization of $\textup{Mat}_{2 \times 2}(K)$ with respect to a different presentation. 
    \end{enumerate}
\end{proposition}

\begin{question}
Is $C_{n+2}^{\trop}$ Morita equivalent to $C_n^{\trop}$?
\end{question}

A motivation for the above question comes from a notorious conjecture of MacPherson.  He asked if the space of oriented matroids is homotopy equivalent to the real Grassmannian.  A somewhat weaker question is:

\begin{question}
Do the spaces of oriented matroids exhibit a form of Bott periodicity? 
\end{question}


\begin{small}
 \bibliographystyle{alpha}
 \bibliography{ref}
\end{small}

\end{document}